\documentclass[reqno]{amsart}

\usepackage{color}

\definecolor{darkgreen}{rgb}{0.00,0.50,0.10}
\definecolor{lightgreen}{rgb}{0.20,0.70,0.30}
\usepackage[colorlinks,bookmarks,linkcolor=darkgreen,citecolor=lightgreen]{hyperref}

\newtheorem{theorem}{Theorem}
\newtheorem{lemma}[theorem]{Lemma} 

\newcommand{\By}[2]{\overset{\mbox{\tiny{#1}}}{#2}}
\newcommand{\eqBy}[1]{    \By{#1}{=} }

\title[Nullstellensatz over integral domains]{
Proof of the combinatorial nullstellensatz over integral domains in the spirit of Kouba
}

\author{Peter Heinig}
\address{Zentrum Mathematik, Lehr- und Forschungseinheit M9 f\"{u}r Angewandte Geometrie und Diskrete Mathematik, 
Technische Universit\"at M\"unchen, Boltzmannstra\ss{}e~3, D-85747 Garching bei M\"unchen, Germany} 
\email{heinig@ma.tum.de}

\thanks{The author was supported by a scholarship from the Max Weber-Programm Bayern and 
by the ENB graduate program TopMath}

\begin{document}
\begin{abstract}
It is shown that by eliminating duality theory of vector spaces from a recent proof of 
Kouba (O. Kouba, A duality based proof of the Combinatorial Nullstellensatz. Electron. J. Combin. 16 (2009), \#N9) 
one obtains a direct proof of the nonvanishing-version of Alon's Combinatorial Nullstellensatz for 
polynomials over an arbitrary integral domain. The proof relies on Cramer's rule and Vandermonde's determinant 
to explicitly describe a map used by Kouba in terms of cofactors of a certain matrix.

That the Combinatorial Nullstellensatz is true over integral domains is a 
well-known fact which is already contained in Alon's work and emphasized in 
recent articles of Micha\l{}ek and Schauz; the sole purpose of the present note 
is to point out that not only is it not necessary to invoke duality of vector 
spaces, but by not doing so one easily obtains a more general result.

\end{abstract}

\maketitle

\section{Introduction}
The Combinatorial Nullstellensatz is a very useful theorem (see ~\cite{AlonNss}) 
about multivariate polynomials over an integral domain which bears some resemblance to the 
classical Nullstellensatz of Hilbert. 

\begin{theorem}[Alon, {C}ombinatorial {N}ullstellensatz (ideal-containment-version), Theorem 1.1 in ~\cite{AlonNss}]\label{thm:idealNss} 
Let $K$ be a field, $R\subseteq K$ a subring,  $f\in R[x_1,\dotsc, x_n]$, $S_1,\dotsc, S_n$ arbitrary nonempty subsets of $K$, 
and $g_i:=\prod_{s\in S_i}(x_i-s)$ for every $1\leq i \leq n$. If $f(s_1,\dotsc, s_n)=0$ 
for every $(s_1,\dotsc, s_n)\in S_1\times\dotsm\times S_n$, then there exist polynomials $h_i\in R[x_1,\dotsc, x_n]$ with the property 
that $\deg(h_i)\leq \deg(f) - \deg(g_i)$ for every $1\leq i \leq n$, and $f=\sum_{i=1}^n h_i g_i$.
\end{theorem}
\begin{theorem}[Alon, {C}ombinatorial {N}ullstellensatz (nonvanishing-version), Theorem 1.2 in ~\cite{AlonNss}]\label{thm:nonvanishingNss} 
Let $K$ be a field, $R\subseteq K$ a subring, and $f\in R[x_1,\dotsc, x_n]$. 
Let $c\cdot x_1^{d_1}\dotsm x_n^{d_n}$ be a term in $f$ with $c\neq 0$ whose degree $d_1+\dotsm+d_n$ is 
maximum among all degrees of terms in $f$. Then every product $S_1\times\dotsm\times S_n$, where each $S_i$ is an arbitrary finite 
subset of $R$ satisfying $|S_i|=d_i+1$, contains at least one point $(s_1,\dotsc, s_n)$ with $f(s_1,\dotsc,s_n)\neq 0$.
\end{theorem}

Three comments are in order. First, talking about subrings of a field is equivalent to talking about integral domains: every 
subring of a field clearly is an integral domain, and, conversely, every integral domain $R$ is (isomorphic to) a subring 
of its field of fractions $\mathrm{Quot}(R)$. Second, strictly speaking, rings are mentioned in ~\cite{AlonNss} only in 
Theorem \ref{thm:idealNss}, but Alon's proof in ~\cite{AlonNss} of Theorem \ref{thm:nonvanishingNss} is valid for polynomials over 
integral domains as well. Third, it is intended that the $S_i$ are allowed to be subsets of $K$ in Theorem ~\ref{thm:idealNss} but required 
to be subsets of $R$ in Theorem ~\ref{thm:nonvanishingNss}, since this is the slightly stronger formulation: if Theorem ~\ref{thm:nonvanishingNss} 
is true as it is formulated here, then by invoking it with $R=K$ and by viewing an $f\in R[x_1,\dotsc,x_n]$, $R$ being a subring of $K$, as a 
polynomial in $K[x_1,\dotsc,x_n]$, it is true as well with the $S_i$ being allowed to be arbitrary subsets of $K$. 

In ~\cite{AlonNss}, Theorem ~\ref{thm:nonvanishingNss} was deduced from Theorem \ref{thm:idealNss}.
In ~\cite{Kouba}, Kouba gave a beautifully simple and direct proof of the 
nonvanishing-version of the Combinatorial Nullstellensatz, bypassing the use of the ideal-containment-version. 
Kouba's argument was restricted to the case of polynomials over a field and at one step applied a suitably 
chosen linear form on the vector space $K[x_1,\dotsc, x_n]$ to the given polynomial $f$ in Theorem ~\ref{thm:nonvanishingNss}. 

However, for Kouba's idea to work, it is not necessary to have recourse to duality theory 
of vector spaces and in the following section it will be shown how to make Kouba's idea work without it, 
thus obtaining a direct proof of the full Theorem ~\ref{thm:nonvanishingNss}. 

Finally, two relevant recent articles ought to be mentioned. A very short direct proof 
of Theorem ~\ref{thm:nonvanishingNss} was given by Micha\l{}ek in ~\cite{Michalek} who explicitly 
remarks that the proof works for integral domains as well. Moreover, the differences
 $\bigr\{s-s':\; \{s,s'\}\in \genfrac{(}{)}{0pt}{}{S_k}{2} \bigl\}$ in the proof below play a similar role in 
Micha\l{}ek's proof. In ~\cite{Schauz}, Schauz obtained far-reaching generalizations and sharpenings 
of Theorem ~\ref{thm:nonvanishingNss}, expressly working with integral domains and generalizations 
thereof throughout the paper.

\section{Proof of Theorem  ~\ref{thm:nonvanishingNss}}
The proof of the Theorem ~\ref{thm:nonvanishingNss} will be based on the following simple lemma.
\begin{lemma}\label{lem:vandermonde} Let $R$ be an integral domain. Let $S=\{s_1,\dotsc,s_m\}\subseteq R$ be an arbitrary 
finite subset. Then there exist elements $\lambda_1^{(S)},\dotsc, \lambda_m^{(S)}$ of $R$ such that 
\begin{gather} \lambda_1^{(S)} \cdot (1, s_1, s_1^2,\dotsc, s_1^{m-1}) + \dotsm + \lambda_m^{(S)} \cdot (1,s_m,s_m^2,\dotsc, s_m^{m-1}) \notag\\
= (0,0,0,\dotsc,0,\prod_{1\leq i < j \leq m} (s_i-s_j)). 
\end{gather}
\end{lemma}
\begin{proof} Let $[m]:=\{1,\dotsc, m\}$. Define $b$ to be the right-hand side of the claimed equation, 
taken as a column vector, and let $A=(a_{ij})_{(i,j)\in [m]^2}$ be the Vandermonde matrix defined by  $a_{ij}:=s_{j}^{i-1}$. 
Then the statement of the lemma is equivalent to the existence of a solution $\lambda^{(S)}\in R^m$ of the system of linear equations 
$A \lambda^{(S)} = b.$ By the well-known formula for the determinant of a 
Vandermonde matrix (see ~\cite{LangAlgebra}, Ch. XIII, \S 4, example after Prop. 4.10), $\det(A)=\prod_{1\leq i < j \leq m} (s_i-s_j)$.

Since $S$ is a set, all factors of this product are nonzero, and since $R$ has 
no zero divisors, the determinant is therefore nonzero as well. Now 
let $\alpha_{ij}$ be the cofactors of $A$, i.e. $\alpha_{ij}:=(-1)^{i+j} \det(A^{(ij)})$, 
where $A^{(ij)}$ is the $(m-1)\times (m-1)$ matrix 
obtained from $A$ by deleting the $i$-th row and the $j$-th column (see ~\cite{BirkhoffMacLane}, Ch. IX, \S 3, before Lemma 1). 
By Cramer's rule (see Ch. IX, \S 3, Corollary 2 of Theorem 6 in ~\cite{BirkhoffMacLane} or Theorem 4.4 in ~\cite{LangAlgebra}), for 
every $j\in [m]$,
\[  \det(A) \cdot \lambda_j^{(S)} = \sum_{i=1}^m \alpha_{ij} b_i.\]
Using $b_m=\det(A)$, $b_i=0$ for every $1\leq i < m$, and the commutativity of an integral domain, this reduces to 
\[ \det(A) \cdot \bigl( \lambda_j^{(S)} - \alpha_{mj}\bigr) = 0.\] 
Hence, since $\det(A)\neq 0$ and $R$ has no zero divisors, if follows that 
the cofactors  $\lambda_j^{(S)}=\alpha_{mj}\in R$ provide explicit elements 
with the desired property.
\end{proof} 
Using this lemma, Kouba's argument may now be carried out without change in the setting of integral domains.
\begin{proof}[Proof of Theorem ~ \ref{thm:nonvanishingNss}] Let $R$ be an arbitrary integral domain and $f\in R[x_1,\dotsc, x_n]$ be 
an arbitrary polynomial. Let $d_1,\dotsc, d_n\in \mathbb{N}_{\geq 0}$ be the exponents of a term $c\cdot  x_1^{d_1}\dotsm x_{n}^{d_n}$ 
with $c\neq 0$ which has maximum degree in $f$. For each $k\in [n]$, choose an arbitrary finite 
subset  $S_k\subseteq R$ and apply Lemma ~\ref{lem:vandermonde} with $S=S_k$ and $m=|S|=d_k+1$ to obtain a family of 
elements $(\lambda_{s_k}^{(S_k)})_{s_k\in S_k}$ of $R$ (where in order to avoid double indices the coefficients $\lambda$ 
are now being indexed by the elements of $S_k$ directly, not by an enumeration of each $S_k$) with the property that
\begin{align}
\sum_{s_k\in S_k} \lambda_{s_k}^{(S_k)} \cdot s_{k}^\ell & = 0\quad \text{ for every $\ell\in \{0,\dotsc, d_k-1\}$}, \label{eq:smallExponent}\\
\sum_{s_k\in S_k} \lambda_{s_k}^{(S_k)} \cdot s_{k}^{d_k} & = \prod_{\{s,s'\} \in \genfrac{(}{)}{0pt}{}{S_k}{2}} (s-s') =: r_k \in R\backslash\{ 0 \}\label{eq:matchingExponents}.
\end{align}

Using the coefficient families $(\lambda_{s_k}^{(S_k)})_{s_k\in S_k}$, define, \`a la Kouba, the map
\begin{align}
  \Phi:\quad & R[x_1,\dotsc, x_n]  \longrightarrow R \notag\\
 g &\longmapsto \sum_{(s_{1},\dotsc,s_{n})\in S_1\times \dotsm \times S_n} \lambda_{s_1}^{(S_1)}\dotsm\lambda_{s_n}^{(S_n)} \cdot g(s_{1},\dotsc,s_{n}).
\label{def:koubaMap}
\end{align}
Due to the commutativity of an integral domain, $\Phi$ is an $R$-linear form on the $R$-module $R[x_1,\dotsc, x_n]$, hence 
its value $\Phi(f)$ on a polynomial $f$ can be evaluated termwise as 
\begin{align}
\Phi(f) = \sum_{c\cdot t\text{ a term in $f$}} c\cdot \Phi(t).
\end{align}
If $t = c\cdot x_1^{d_1'}\dotsm x_n^{d_n'}$ is an arbitrary term in $R[x_1,\dotsc,x_n]$, then
\begin{align}
\Phi(t) = c\cdot \Phi (x_1^{d_1'}\dotsm x_n^{d_n'}) 
& = c \cdot\sum_{(s_{1},\dotsc,s_{n})\in S_1\times \dotsm \times S_n} \lambda_{s_1}^{(S_1)}\dotsm\lambda_{s_n}^{(S_n)} \cdot  s_{1}^{d_1'} \dotsm s_{n}^{d_n'} \notag\\
& = c \cdot \sum_{s_{1}\in S_1}\dotsm \sum_{s_{n} \in S_n} \lambda_{s_1}^{(S_1)}\dotsm\lambda_{s_n}^{(S_n)} \cdot s_{1}^{d_1'} \dotsm s_{n}^{d_n'} \notag\\
& = c \cdot \prod_{k=1}^n\biggl ( \sum_{s_k \in S_k} \lambda_{s_k}^{(S_k)} s_k^{d_k'},\biggr )\label{eq:productExpansion}
\end{align}
where in the last step again use has been made of the commutativity of an integral domain. By ~\eqref{eq:productExpansion} 
and ~\eqref{eq:smallExponent} it follows that for every term $t$, if there is at least one exponent $d_i'$ with $d_i'<d_i$, 
then $\Phi(t)=0$. Moreover, by the choice of the term $c\cdot x_1^{d_1}\dotsm x_{n}^{d_n}$, every 
term $c'\cdot x_1^{d_1'}\dotsm x_{n}^{d_n'}$ of $f$ which is different from the term $c\cdot x_1^{d_1}\dotsm x_{n}^{d_n}$ must, 
even if it has itself maximum degree in $f$, contain at least one exponent $d_i'$ with $d_i'<d_i$. Therefore
\begin{align}
& \sum_{(s_{1},\dotsc,s_{n})\in S_1\times \dotsm \times S_n} \lambda_{s_1}^{(S_1)}\dotsm\lambda_{s_n}^{(S_n)} \cdot f(s_{1},\dotsc,s_{n}) 
\eqBy{\eqref{def:koubaMap}} 
\Phi(f) \eqBy{\eqref{eq:smallExponent},\eqref{eq:productExpansion}}
 c\cdot \Phi(x_1^{d_1}\dotsm x_{n}^{d_n} ) = \notag\\
&  \eqBy{\eqref{eq:matchingExponents},\eqref{eq:productExpansion}} 
c\cdot \prod_{k=1}^n \prod_{\{s,s'\} \in \genfrac{(}{)}{0pt}{}{S_k}{2}} (s-s') = c\cdot \prod_{k=1}^n r_k \neq 0,
\end{align}
since $R$ has no zero divisors. Obviously this implies that there exists at least one point $(s_1,\dotsc, s_n)\in S_1\times\dotsm \times S_n$ where 
$f$ does not vanish.
\end{proof}

\section{Concluding question}
Is there any interesting use for the fact that even in the 
case of integral domains the coefficients of Kouba's map can be explicitly 
expressed in terms of cofactors of the matrices $(s^{i-1}_j)$?

\section*{Acknowledgement} The author is very grateful to the department M9 of 
Technische Universit\"{a}t M\"{u}nchen for excellent working conditions.

\bibliographystyle{amsplain} 
\bibliography{HEINIGnullstellensatzInSpiritOfKouba}

\providecommand{\bysame}{\leavevmode\hbox to3em{\hrulefill}\thinspace}
\providecommand{\MR}{\relax\ifhmode\unskip\space\fi MR }
\providecommand{\MRhref}[2]{%
  \href{http://www.ams.org/mathscinet-getitem?mr=#1}{#2}
}
\providecommand{\href}[2]{#2}
\begin{thebibliography}{1}

\bibitem{AlonNss}
N.~Alon, \emph{{C}ombinatorial {N}ullstellensatz}, Combin. Probab. Comput.
  \textbf{8} (1999), no.~1, 7--29.

\bibitem{BirkhoffMacLane}
G.~D. Birkhoff and S.~Mac~Lane, \emph{Algebra}, 3. ed., American Mathematical
  Society, 1987.

\bibitem{Kouba}
O.~Kouba, \emph{A duality based proof of the {C}ombinatorial
  {N}ullstellensatz}, Electron. J. Combin. \textbf{19} (2009), \#N9.

\bibitem{LangAlgebra}
S.~Lang, \emph{Algebra}, 3. ed., Graduate Texts in Mathematics, vol. 211,
  Springer, 2002.

\bibitem{Michalek}
M.~Micha{\l}ek, \emph{A short proof of {C}ombinatorial {N}ullstellensatz},
  arXiv:0904.4573v1 [math.CO] (2009).

\bibitem{Schauz}
U.~Schauz, \emph{Algebraically {S}olvable {P}roblems: {D}escribing
  {P}olynomials as {E}quivalent to {E}xplicit {S}olutions}, Electron. J.
  Combin. \textbf{15} (2008), \#R10.

\end{thebibliography}


\end{document}